\documentclass[11pt]{article}
\usepackage{amsfonts,amsmath, amssymb,latexsym}
\usepackage[OT2,OT1]{fontenc}
\def\SH{\mbox{\fontencoding{OT2}\selectfont\char88}}

\def\Z{{\mathbb Z}}

\def\Sel{{\rm Sel}}

\def\GL{{\rm GL}}

\def\rk{{\rm rk}}
\def\an{{\rm an}}

\def\Gal{{\rm Gal}}

\def\equi{{\rm equi}}

\def\spl{{\rm split}}
\def\ns{{\rm nonsplit}}
\def\g{{\rm good}}

\def\F{{\mathbb F}}

\def\Q{{\mathbb Q}}

\def\Z{{\mathbb Z}}

\def\F{{\mathbb F}}
\def\Q{{\mathbb Q}}
\def\Qp{{\Q_p}}
\def\Zp{{\Z_p}}
\def\bQ{{\bar\Q}}

\def\grL{{\mathfrak L}}
\def\ord{{\mathrm{ord}}}

\newtheorem{theorem}{Theorem}
\newtheorem{corollary}[theorem]{Corollary}
\newtheorem{lemma}[theorem]{Lemma}

\newtheorem{remark}[theorem]{Remark}
\newenvironment{proof}{\noindent {\bf Proof:}}{$\Box$ \vspace{2 ex}}

\title{A majority of elliptic curves over $\Q$ satisfy the \\Birch and Swinnerton-Dyer conjecture}

\author{Manjul Bhargava, Christopher Skinner, and Wei Zhang}

\begin{document}
\maketitle
\begin{abstract}
We prove that a majority (in fact, $>66\%$) of all elliptic curves over $\Q$, when ordered by height, satisfy the Birch and Swinnerton-Dyer rank conjecture. 
\end{abstract}

\tableofcontents

\section{Introduction}

Any elliptic curve $E$ over $\Q$ is isomorphic to a unique curve of
the form $E_{A,B}:y^2=x^3+Ax+B$, where $A,B \in \Z$ and for all primes
$p$:\, $p^6 \nmid B$ whenever $p^4 \mid A$. The (naive) {\it height}
$H(E_{A,B})$ of the elliptic curve $E=E_{A,B}$ 
is then defined by
$$H(E_{A,B}):= \max\{4|A^3|,27B^2\}.$$
The purpose of this article is to prove the following theorem:

\begin{theorem}\label{bsdcor}
A majority of elliptic curves over $\Q$, when ordered by height, satisfy the Birch and Swinnerton-Dyer rank conjecture. 
\end{theorem}
As a consequence of our methods, we also obtain:

\begin{theorem}\label{shacor}
A majority of elliptic curves over $\Q$, when ordered by height, have finite Tate--Shafarevich group. 
\end{theorem}
In fact, we will prove that the words ``A majority'' in Theorems~\ref{bsdcor} and \ref{shacor} can be replaced with ``At least $66.48\%$''.  

More precisely, for any elliptic curve $E$ over $\Q$, let us denote by          
$\rk(E)$ the algebraic rank of~$E$, i.e., the rank of the Mordell--Weil         
group $E(\Q)$. Let $\rk_\an(E)$ denote the analytic rank of~$E$, i.e., 
the order of vanishing at the central critical point of the $L$-function of $E$
(an entire function by the modularity theorem 
\cite{W,TW,BCDT}). Let $\SH(E)$ denote the Tate--Shafarevich group of~$E$. Then we prove
\begin{eqnarray}\label{precisethm}
{\liminf_{X\to\infty}}\,
\frac
{\#\{E /\Q:
\rk(E)=\rk_\an(E),\mbox{ $\SH(E)$ is finite, and } H(E)<X \}}
{\#\{E /\Q: H(E)<X\}}&>&66.48\%,
\end{eqnarray}
although 
this percentage can likely be improved with a more careful analysis (see, e.g., the last paragraph of this introduction).

We also obtain new lower bounds for the proportion of elliptic curves having both algebraic and analytic rank zero, and for the proportion having both algebraic and analytic rank one:

\begin{theorem}\label{bsdrank0and1} 
When ordered by height, at least $16.50\%$ of elliptic curves over $\Q$ have algebraic and analytic rank zero, and at least $20.68\%$ of elliptic curves have algebraic and analytic rank one.
\end{theorem}
As a consequence, we also obtain a new lower bound on the ($\liminf$ of the) average rank of all elliptic curves, when ordered by height:

\begin{corollary}\label{avecor}
The average $($algebraic or analytic$)$ rank of elliptic curves over $\Q$, when ordered by height, is at least $.2068$.
\end{corollary}

The observant reader will have noticed that the sum of the percentages in Theorem~\ref{bsdrank0and1} is less than the $66.48\%$ stated above for the majority in Theorems~\ref{bsdcor} and \ref{shacor}. This is because our method, described below, enables us to prove better lower bounds on the proportion of curves that have analytic rank 0 {\it or} 1 than for the sum of individual lower bounds on the proportions of curves having analytic rank 0 and  those having analytic rank 1!  

The proofs of Theorems~\ref{bsdcor},~\ref{shacor}, and~\ref{bsdrank0and1} make combined use of results 
from~\cite{BS3}, \cite{BS5}, \cite{DD}, \cite{Smult}, \cite{SU},  \cite{SZ}, and \cite{Z}.
The works \cite{Smult} and \cite{SU} allow one to deduce sufficient $p$-adic conditions for an elliptic curve to have algebraic and analytic rank zero, assuming $p\geq 3$.  
The works \cite{SZ} and \cite{Z} (which also rely on \cite{Smult} and \cite{SU}, as well as on 
the Gross--Zagier formula \cite{GZ}, its extension to Shimura curves by Shouwu Zhang \cite{SZ-GZ},
Kolyvagin's theory of Euler systems of Heegner points~\cite{Koly-Euler}, and    
the variant described by Bertolini--Darmon~\cite{BD-ACMC}) give sufficient 
$p$-adic conditions for an elliptic curve to have algebraic and analytic rank one, assuming $p\geq 5$. 
In particular, one can deduce certain sufficient conditions for an elliptic curve to have algebraic and analytic rank zero or one, in accordance with whether
its $p$-Selmer group has rank zero or one. We deduce two such theorems---namely, Theorems~\ref{crit0-thm1} and~\ref{crit1-thm1}---in~\S2.
On the other hand, the works \cite{BS3} and \cite{BS5} allow one to determine the average sizes of $p$-Selmer groups, for $p=3$ and $5$, respectively, in any 
{\it large} family of elliptic curves defined by congruence conditions (see \S\ref{selavg} for the definition of a large family).  

We combine these results as follows. Aside from requiring the $p$-Selmer groups to have rank zero or one, the conditions required of the elliptic curves in 
Theorems~\ref{crit0-thm1}~and~\ref{crit1-thm1} are either congruence conditions (e.g.,~having good ordinary or multiplicative reduction at $p$) or conditions that hold \pagebreak for $100\%$ of elliptic curves (e.g.,~$E[p]$ being an irreducible Galois representation).  In some cases, infinitely many congruence conditions are required, but nevertheless the families are large in the sense of \S\ref{selavg}.  We then estimate
the density of the elliptic curves in the large families defined by the congruence conditions in these theorems. In both cases, the density is shown to be close 
to~$80\%$. Combined with the results from \cite{BS3}~and~\cite{BS5} on average sizes of $p$-Selmer groups in large families, this is then already enough to deduce that  a sizable proportion of curves have algebraic and analytic rank zero or algebraic and analytic rank one. 

However, to deduce that there are many curves of analytic rank zero and many curves of analytic rank one, we need to know that there are a large number of curves with root number $+1$ and root number  $-1$. In fact, using the results of \cite{BS5}, we show that each of the large families of elliptic curves that we consider contains a large subfamily of density at least $55\%$ 
for which the root numbers are equidistributed, when the curves are ordered by height. Using this and the theorem of Dokchitser and Dokchitser~\cite{DD} on the correspondence between root numbers and $p$-Selmer parity, and then optimizing the combinatorics, enables us to deduce that many elliptic curves have algebraic and analytic rank zero, many have algebraic and analytic rank one, and most have algebraic and analytic rank zero {\it or} algebraic and analytic rank one,  yielding Theorems~\ref{bsdcor} and~\ref{bsdrank0and1}, and Corollary~\ref{avecor}. 
Since we show that the majority of elliptic curves over $\Q$ have analytic rank at most one, we then obtain Theorem~\ref{shacor} through the work of Kolyvagin {\it et al.}~proving the finiteness of the Tate-Shafarevich group in these cases.

In the earlier papers \cite{BS3} and \cite{rankone}, 
it was shown that a positive proportion of elliptic curves, when ordered by height, have 
algebraic and analytic rank zero, and a positive proportion have algebraic and analytic rank one, respectively.
In the rank one case, 
the proof relied on a different 
$p$-adic criterion, proved in \cite{S}, for an elliptic curve to have rank one, together with an equidistribution result on the local
images of Selmer elements. We expect that forthcoming work that combines some of the techniques of \cite{S} and \cite{Z} and extends 
the results to cases of supersingular reduction will substantially improve the percentages given in this paper (see~\S4).  It is also worth noting that if 
the Selmer average results of \cite{BS2}, \cite{BS3}, and~\cite{BS5} could be extended to infinitely many primes, then the words ``A majority'' in 
Theorems~\ref{bsdcor}~and~\ref{shacor} could be replaced with ``$100\%$'' (see~\S4, Theorem~\ref{allpavg}).

\section{Preliminaries}

In this section, we collect the results that we need as ingredients for the proofs of Theorems~\ref{bsdcor}--\ref{bsdrank0and1}.

\subsection{$p$-adic conditions for an elliptic curve to have algebraic and analytic rank 0}

We begin with the following criterion, deduced from results in \cite{SU} and \cite{Smult}, which gives a sufficient condition for
an elliptic curve over $\Q$ to have algebraic and analytic rank zero:

\begin{theorem}\label{crit0-thm1}
Let $p$ be an odd prime. Let $E$ be an elliptic curve over $\Q$ with conductor $N$ such that:
\begin{itemize}
\item[{\rm (a)}] $E$ has good ordinary or multiplicative reduction at $p$;
\item[{\rm (b)}] $E[p]$ is an irreducible $\Gal(\bQ/\Q)$-module;
\item[{\rm (c)}] there is at least one prime $\ell\neq p$ such that $\ell\mid\mid N$ and $E[p]$ is ramified at $\ell$;
\item[{\rm (d)}] the $p$-Selmer group $S_p(E)$ of $E$ is trivial.
\end{itemize}
Then the rank and analytic rank of $E$ are both equal to $0$.
\end{theorem}

\begin{proof} If the $p$-Selmer group $S_p(E)$ of $E$ is trivial, then so is the $p^\infty$-Selmer group (denoted
$\Sel_{p^\infty}(E/\Q)$ and $\mathrm{Sel}_{p^\infty}(E)$ in \cite{SU} and \cite{Smult}, respectively), and hence the 
Mordell--Weil group $E(\Q)$ is finite. 
It also follows from \cite[Thm.~2(b)]{SU} (for the case of good reduction at $p$) and
\cite[Thm.~B]{Smult} (for the case of multiplicative reduction at $p$) that under the hypotheses on $E$ stated in the theorem, we have $L(E,1)\neq 0$.~\end{proof}

\begin{remark}\label{rmk0}{\em 
Work in progress 
\cite{SU2} should also allow $E$ to have good supersingular reduction, remove the condition (b),
and replace (c) with just the existence of one prime $\ell\neq p$ at which $E$ has potentially multiplicative or supercuspidal reduction.
 }\end{remark}
 
 \begin{remark}\label{rmk-ram}{\em 
 Suppose $\ell\mid\mid N$ and $\ell\neq p$.  The condition that $E[p]$ be ramified at $\ell$ is equivalent to $p\nmid\ord_\ell(\Delta_\ell)$ for 
 a minimal discriminant $\Delta_\ell$ of $E$ at $\ell$ \cite[Prop.~V.6.1 \& Ex.~V.5.13]{Silverman2}. 
 In particular, for $\ell\geq 5$, a Weierstrass equation $E_{A,B}:y^2=x^3+Ax+B$ that satisfies
 $\ell^6\nmid B$ if $\ell^4\mid A$ is minimal at $\ell$, and so if $\ell\mid\mid N$ then $E_{A,B}[p]$ is ramified at $\ell$ if and only if 
 $p\nmid \ord_\ell(\Delta(E_{A,B}))$.  This will be used later to interpret hypothesis (c) of Theorem~\ref{crit0-thm1} (and hypotheses (c) and (d) 
 of Theorem~\ref{crit1-thm1} below) as congruence conditions at $\ell$.}
 \end{remark}
 
 \begin{remark}\label{rmk-rank0}{\em 
 Both \cite[Thm.~2(b)]{SU} and \cite[Thm.~B]{Smult}, on which the proof of Theorem~\ref{crit0-thm1} relies,
 identify the power of $p$ dividing the product of the order of the $p^\infty$-Selmer group $\Sel_{p^\infty}(E)$ of $E$ and the Tamagawa factors of $E$
 as the power of $p$ dividing $L(E,1)/\Omega_E$, where $\Omega_E$ is the real period of~$E$. As explained in \cite{Smult},
 this relation is a consequence of the Iwasawa--Greenberg main conjecture for~$E$ (which is the main
 result of \cite{SU} and \cite{Smult}). One consequence of this equality of powers of $p$ is that if $\Sel_{p^\infty}(E)$ is finite, then $L(E,1)\neq 0$.  To conclude further that the
 Tate--Shafarevich group~$\SH(E)$ of~$E$ is finite then requires an appeal to the work of Gross--Zagier~\cite{GZ} and Kolyvagin~ \cite[Thm.~A and Cor.~B]{Koly-Euler}.}
 \end{remark}

\subsection{$p$-adic conditions for an elliptic curve to have algebraic and analytic rank 1}

Next, for $p\geq 5$ we 
have sufficient $p$-adic conditions, deduced from \cite{Z} and \cite{SZ},
for an elliptic curve over $\Q$ to have algebraic and analytic rank one:

\begin{theorem}\label{crit1-thm1}
Let $p\geq 5$ be a prime. Let $E$ be an elliptic curve over $\Q$ with conductor $N$ such that:
\begin{itemize}
\item[{\rm (a)}]
$E$ has good ordinary or multiplicative reduction at $p$; 
\item[{\rm (b)}]
$E[p]$ is an irreducible $\Gal(\bar\Q/\Q)$-module;
\item[{\rm (c)}] for all primes $\ell\mid\mid N$ such that $\ell\equiv \pm 1\pmod p$, $E[p]$ is ramified at $\ell$;
\item[{\rm (d)}]  if $N$ is not squarefree, then
 there exist at least two prime factors $\ell\mid\mid N$ with $\ell\neq p$ and where $E[p]$ is ramified;
\item[{\rm (e)}] if $E$ has multiplicative reduction at $p$ then $E[p]$ is not finite at $p$, and if $E$ has split multiplicative
reduction at $p$ then the $p$-adic Mazur--Tate--Teitelbaum $\grL$-invariant $\grL(E)$ of $E$ satisfies $\ord_p(\grL(E))=1$;
\item[{\rm (f)}]
the $p$-Selmer group $S_p(E)$ of $E$ has order $p$.
\end{itemize}
Then the rank and analytic rank of $E$ are both equal to~$1$.
\end{theorem}

\begin{proof} By a theorem of Cassels, the $p^\infty$-Selmer group of $E$ is isomorphic 
as a $\Zp$-module to $(\Qp/\Zp)^r\oplus F\oplus F$ for some integer $r$ (expected to be the rank of $E(\Q)$) and a finite group $F$.
Since $E(\Q)[p]$ is trivial under (b), the $p$-Selmer group $S_p(E)$ is the $p$-torsion of the $p^\infty$-Selmer group. Hence if $S_p(E)$ has order $p$, then $F=0$ and $r=1$.
It follows from this observation and \cite[Thm.~1.3]{Z} (for the case of good reduction) and \cite[Thm.~1.1]{SZ} (for the case of
multiplicative reduction) that under the hypotheses stated for $E$ in the theorem, we have $\ord_{s=1}L(E,s)=1$ and the rank of $E(\Q)$ is 1. \end{proof}

\begin{remark}\label{rmk1}{\em
We expect that combining the methods used to prove Theorem \ref{crit1-thm1} with the approach in \cite{S} should allow $E$ to have 
supersingular reduction at $p$. Furthermore, it should be possible to extend this result to $p=3$ under the additional restrictions
of good reduction at $p$ and that the image of the local restriction $S_p(E)\rightarrow E(\Q_p)/pE(\Q_p)$ not lie in the image of $E(\Q_p)[p]$. 
These extensions require corresponding extensions of the main results in \cite{Wan}.
}\end{remark}

\begin{remark}\label{rmk-finite}{\em We recall that $E[p]$ is ``finite at $p$'' if as a $\Gal(\bQ_p/\Qp)$-representation, 
$E[p]$ is isomorphic to the representation on the $\bQ_p$-points of a finite flat group scheme over $\Zp$ (this is always the case if
$E$ has good reduction at $p$). If $E$ has multiplicative reduction at $p$, then $E[p]$ is finite at $p$ if and only if $p\nmid\ord_p(\Delta_p)$
for a minimal discriminant $\Delta_p$ at $p$ \cite[\S2.9,~Prop.~5]{Serre-Duke}. In particular, for $p\geq 5$, a Weierstrass equation $E_{A,B}:y^2=x^3+Ax+B$ that satisfies
 $p^6\nmid B$ if $p^4\mid A$ is minimal at $p$, and so if $p\mid\mid N$ then $E_{A,B}[p]$ is finite at $p$ if and only if 
 $p\nmid \ord_p(\Delta(E_{A,B}))$.  This will be used later to interpret hypothesis (e) of Theorem~\ref{crit1-thm1} as congruence conditions at $p$.}
\end{remark}

\begin{remark}\label{rmk-rank1}{\em
The theorems \cite[Thm.~1.3]{Z} and \cite[Thm.~1.1]{SZ}, on which the proof of Theorem~\ref{crit1-thm1} relies, are consequences of the main 
results of \cite{Z} and \cite{SZ}, respectively. These results show that for an elliptic curve~$E$ as in Theorem \ref{crit1-thm1}, the indices of certain Heegner points on $E$ over appropriate imaginary quadratic
fields are not divisible by $p$ (in particular these Heegner points are non-torsion!). The proof of this indivisibility further relies on comparing \cite[Thm.~2(b)]{SU} and 
\cite[Thm.~B]{Smult} with a special value formula of Gross \cite[(7-8)]{Vatsal-GZ}, and owes much to ideas of Bertolini and Darmon~\cite{BD-ACMC}. 
These Heegner points do not necessarily come from uniformizations by modular curves as in the work of Gross and Zagier \cite{GZ}, but possibly only from uniformizations by Shimura curves,
which is treated in the work of Shouwu Zhang \cite{SZ-GZ}. That $E$ then has both rank and analytic rank~one and finite Tate--Shafarevich group follows from the main results
of \cite{SZ-GZ}.  We note in passing that once it is known that $E$ has analytic rank~one, then it follows that the Heegner points on $E$ in the setting 
of \cite{GZ} are also non-torsion, but they may have indices divisible by $p$.
}\end{remark}

\subsection{Average sizes of $p$-Selmer groups of elliptic curves in large families}\label{selavg}

For $p\in\{3,5\}$, it is easy to see that, when ordered by height, a large proportion of elliptic curves already satisfy conditions (a)--(c) of Theorem~\ref{crit0-thm1} or conditions (a)--(e) of Theorem~\ref{crit1-thm1}.  To show that a positive proportion of elliptic curves also satisfy the hypotheses (d) and (f) of Theorems~\ref{crit0-thm1} and \ref{crit1-thm1}, respectively, we require the following results on average orders of $p$-Selmer groups.

\begin{theorem}[\cite{BS2,BS3,BS5}]\label{thBS}
Let $p\leq5$ be a prime.  When elliptic curves over $\Q$ in any large family are ordered by height, the average size of the $p$-Selmer group is $p+1$.  
\end{theorem} 
We recall what is meant by a ``large family'' of elliptic curves.
For each prime~$\ell$, let
$\Sigma_\ell$ be a closed subset of $\{(A,B)\in\Z_\ell^2 :
\Delta(A,B):=-4A^3-27B^2\neq 0\}$ with boundary of measure $0$.
To such a collection $\Sigma=(\Sigma_\ell)_\ell$,
we associate the set $F_\Sigma$ of elliptic curves over $\Q$, where
$E_{A,B}\in F_\Sigma$ if and only if $(A,B)\in\Sigma_\ell$ for all~$\ell$.
We then say that $F_\Sigma$ is a family of elliptic curves over~$\Q$ 
that is {\it defined by congruence conditions}.
(Although we shall not do so in this article, we can also impose ``congruence conditions at infinity'' on $F_\Sigma$, by
insisting that an elliptic curve $E_{A,B}$ belongs to $F_\Sigma$ if
and only if $(A,B)$ belongs to $\Sigma_\infty$, where $\Sigma_\infty$
consists of all $(A,B)$ with $\Delta(A,B)$ positive, or negative, or either.)

A family $F=F_\Sigma$ of elliptic curves defined
by congruence conditions is said to be {\it large} if, for all
sufficiently large primes $\ell$, the set $\Sigma_\ell$ contains all
$E_{A,B}$ with $(A,B)\in\Z_\ell^2$ such that $\ell^2\nmid\Delta(A,B)$. For
example, the family of all elliptic curves is large, as is any
family of elliptic curves $E_{A,B}$ defined by finitely many congruence
conditions on $A$ and $B$.  The family
of all semistable elliptic curves is also large. Any large family
makes up a positive proportion of all elliptic curves~\cite[Thm.~3.17]{BS2}.

\begin{remark}{\em 
Theorem~\ref{thBS} is proven in \cite{BS2,BS3,BS5} by applying techniques from the geometry-of-numbers and sieve methods to count integral models of
$p$-Selmer group elements---as binary quartic forms, ternary cubic forms, and quintuples of quinary alternating 2-forms---for $p=2$, 3, and 5, respectively.  The~theory
of such integral models was developed in works of Birch--Swinnerton-Dyer~\cite{BSD}, Cremona--Fisher--Stoll~\cite{CFS}, and Fisher~\cite{Fishermin}, respectively. 
We note that a refinement of Theorem~\ref{thBS}, establishing the equidistribution of the local images in $E(\Q_\ell)/pE(\Q_\ell)$ of $p$-Selmer group elements of elliptic curves $E$ lying in sufficiently small $\ell$-adic discs, was given in \cite{rankone}.    
Such a refinement may be useful in future improvements to Theorem~\ref{bsdcor}, particularly when using $p=3$ (see Remark~\ref{rmk1}).
}\end{remark}

\subsection{Root numbers and Selmer parity}

Theorem \ref{thBS} alone is not sufficient to guarantee the existence of curves satisfying either Condition~(d) of Theorem~\ref{crit0-thm1} or Condition~(f) of Theorem~\ref{crit1-thm1}.   
In order to deduce positive proportion statements for the
individual ranks~0 and~1, we will make use of information
regarding the distribution of the {\it parity} of the ranks---or of
the $p$-Selmer ranks---of these curves.  In this direction, 
in \cite{BS5}, a union $\mathcal F$ of large families was constructed in which half of the members of $\mathcal F$ have root number $+1$ and half $-1$.  The following theorem of Dokchitser--Dokchitser~\cite{DD} 
(see also Nekov\'a\v{r}~\cite{Nekovar}) then allows us to relate these root numbers with $p$-Selmer ranks. 

\begin{theorem}[Dokchitser--Dokchitser]\label{thDD}
Let $E$ be an elliptic curve over $\Q$ and let $p$ be any prime. 
Let $s_p(E)$ and $t_p(E)$ denote the rank of the $p$-Selmer group of $E$ and the rank of $E(\Q)[p]$, respectively.
Then the quantity $r_p(E):=s_p(E)-t_p(E)$ is even if and only if the root
number of $E$ is~$+1$.
\end{theorem}
The theorem of Dokchitser and Dokchitser, together with the following theorem which follows from \cite[\S5]{BS5}, will allow us to deduce that many elliptic curves have even (resp.\ odd) $p$-Selmer parity.

\begin{theorem}\label{equithm}
Let $F$ be any large family of elliptic curves over $\Q$ defined by congruence conditions modulo powers of primes $p$ such that 
$p\equiv 1\pmod 4$. Then there exists a finite union $F'$ of large subfamilies of $F$ such that, when elliptic curves in~$F$ and~$F'$ are ordered by height, the root numbers in $F'$ are equidistributed and $F'$ contains a density of greater than $55.01\%$ of the elliptic curves in~$F$. 
\end{theorem}
In particular, Theorems~\ref{thDD} and \ref{equithm} together imply that, for any prime $p$, a density of greater than $55.01\%$ of all elliptic curves over $\Q$ have equidistributed parity of $p$-Selmer rank.  
By construction, the elliptic curves $E\in F'$ in Theorem~\ref{equithm} 
have the property that $E$ and its $-1$-twist $E_{-1}$ both lie in $F'$ and, moreover, have root numbers of opposite sign.

\section{Proofs of Theorems \ref{bsdcor} and \ref{shacor}}

\subsection{Setup and notation}

For a large family $F$ of elliptic curves $E_{A,B}$, let $\mu(F)$ denote the density of all elliptic curves over~$\Q$, when ordered by height, lying in $F$.  When $F=F_\Sigma$ is a large family defined by congruence conditions, by \cite[Thm.~3.17]{BS2} the density 
$\mu(F)$ equals the product of the local densities $\mu_\ell(F)$ over all primes $\ell$, where $\mu_\ell(F) = \mu_\ell(\Sigma_\ell)$ 
equals the measure of 
$\Sigma_\ell\subset\Z_\ell^2$ divided by $1-\ell^{-10}$; here $1-\ell^{-10}$ is the measure of the set of $(A,B)\in\Z_\ell^2$ such that
$\ell^6\nmid B$ whenever~$\ell^4\mid A$.

We begin
this section by estimating the densities of three large families, $S_0(5)$, $S_1'(5)$,  and $S_1(5)$,
that satisfy many of the conditions of Theorems~\ref{crit0-thm1}~or~\ref{crit1-thm1} for the prime $p=5$.
These families are defined as follows. 
For any prime $p\geq 5$,
let $S_0(p)$ be the set of elliptic curves $E_{A,B}$ over $\Q$ such that: 
\begin{itemize} 
\item $E_{A,B}$ has good ordinary or multiplicative reduction at $p$;
\end{itemize}
let $S_1'(p)$ be the subset of curves $E_{A,B}\in S_0(p)$ also satisfying:
\begin{itemize}
\item if $E_{A,B}$ has multiplicative reduction at $p$, then $p\nmid \ord_p(\Delta(A,B))$,
\item if $E_{A,B}$ has split multiplicative reduction at $p$, then $\ord_p(\grL(E_{A,B}))=1$;
\end{itemize}
and let $S_1(p)$ be the subset of curves $E_{A,B}\in S_1'(p)$ also satisfying:
\begin{itemize}
\item $p\nmid \ord_\ell(\Delta(A,B))$ for all primes $\ell\equiv\pm 1\!\pmod p$ such that $\ord_\ell(\Delta(A,B))>0$.
\end{itemize}
Then $S_0(p)$ is the set of curves satisfying (a) of Theorem~\ref{crit0-thm1}; $S_1'(p)$ is contained in the set of curves satisfying (a) and (e) of Theorem~\ref{crit1-thm1} (see Remark~\ref{rmk-finite}); and $S_1(p)$ is contained in the set 
of curves satisfying (a), (c), and (e) of Theorem~\ref{crit1-thm1} (see Remark~\ref{rmk-ram}).  All three of these sets $S_0(p)\supset S_1'(p)\supset S_1(p)$ are large families; this will be demonstrated during the course of the proofs of Lemmas~\ref{S0-density},
\ref{S1prime-density}, and~\ref{S1-density}.

For any large family $F$, let $\mu_\equi(F)$ denote the supremum of $\mu(F')$ over all finite unions~$F'$ of large subfamilies of~$F$ for which the root numbers of elliptic curves in $F'$ are equidistributed.  We conjecture that for any large family $F$, we have $\mu_\equi(F)=\mu(F)$.  Although this seems difficult to prove at the moment, we are nevertheless able to show
that for both of the sets 
$F=S_0(5)$ or~$S_1'(5)$,
we have $\mu_\equi(F) = \kappa \cdot \mu(F)$ with $\kappa\geq .5501$;
this will follow from Theorem~\ref{equithm} together with the fact (demonstrated in \S3.2) 
that $S_0(5)$ and~$S_1'(5)$ are large
families defined by congruence conditions modulo powers of~5.

\subsection{The densities of $S_0(5)$, $S_1'(5)$, and $S_1(5)$}

The next three lemmas give the densities of $S_0(5)$, $S_1'(5)$, and $S_1(5)$  among all elliptic curves $E$ over~$\Q$.

\begin{lemma}\label{S0-density} 
We have $\mu(S_0(5))=\displaystyle \frac{4\cdot 5^{10}}{5(5^{10}-1)}> .8$.
\end{lemma}

\begin{proof}
Let $p\geq 5$ be a prime.
It follows from \cite[Rem.~VII.1.1]{Sil} that $E_{A,B}$ is a minimal Weierstrass equation at $p$. Therefore,
$E_{A,B}$ has good reduction at $p$ if and only if $p\nmid \Delta(A,B)$ \cite[Prop.~VII.5.1(a)]{Sil}.
 In this case, the condition that $E_{A,B}$ has ordinary reduction is that $p$ not divide the coefficient
 of $x^{p-1}$ in $(x^3+Ax+B)^{(p-1)/2}$ \cite[Thm.~V.4.1(a)]{Sil}. 
 Similarly, $E_{A,B}$ has multiplicative reduction at $p$ if and only if $p\mid\Delta(A,B)$ but $p\nmid A$
 \cite[Prop.~VII.5.1(b)]{Sil}.  These are just congruence conditions modulo $p$, and so $S_0(p)$
 is a large family. If $p=5$, then these conditions together amount to
 $$
 5\nmid A.
 $$
 The measure of the set of $(A,B)\in\Z_5^2$ satisfying this condition is just $\frac{4}{5}$, while the measure of the set of $(A,B)\in\Z_5^2$
 satisfying $5^6\nmid B$ whenever $5^4\mid A$ is $1-\frac{1}{5^{10}}$. It follows that 
 $$
 \mu(S_0(5)) = \frac{4}{5}\cdot (1-\frac{1}{5^{10}})^{-1} = \frac{4\cdot 5^{10}}{5(5^{10}-1)}.
 $$
 \hfill \end{proof}
 
 \begin{lemma}\label{S1prime-density}
We have $\displaystyle \mu(S_1'(5)) = 
 \left( \frac{99}{125} - \frac{19}{125}\cdot\frac{4}{5^5-1}\right) \left(1-\frac{1}{5^{10}}\right)^{-1} = .7918054\ldots .$
\end{lemma}

\begin{proof} 
Let $p\geq 5$ be a prime. As noted in the proof of Lemma \ref{S0-density}, the Weierstrass equation $E_{A,B}$ is minimal at $p$, 
and $E_{A,B}$ has multiplicative
reduction at $p$ if and only if $p\mid\Delta(A,B)$ but $p\nmid A$. We claim that if $E_{A,B}$ has multiplicative reduction at $p$, 
then it has split
multiplicative reduction if and only if $u^4\equiv -3A \pmod p$ has a solution.
To see this, we work over $\F_p$. 
Let $x_0\in \F_p$ be the double root
of $x^3+Ax+B$, so $3x_0^2 = -A$. The equations for the tangent lines at the nodal point $(x_0,0)$ on the singular curve $E_{A,B}/\F_p$ are
$y-\alpha(x-x_0)=0$ and $y-\beta(x-x_0)=0$ for some $\alpha,\beta\in \F_{p^2}$, and satisfy 
$$
y^2-x^3-Ax-B = (y-\alpha(x-x_0))\cdot (y-\beta(x-x_0)) - (x-x_0)^3.
$$
Equating the respective coefficients of $x$ and $x^2$ on each side, we find that $\alpha+\beta =0$ and $2\alpha\beta x_0 +3x_0^2=A$.
From this, together with the equality $3x_0^2 = -A$, we conclude that $-\alpha^2x_0 = A$. Squaring and again using that $3x_0^2=-A$, 
we see that $\alpha^4 = -3A$.  As $E_{A,B}$ has split multiplicative reduction if and only if $\alpha,\beta \in \F_p$, the claim  follows. 

Recall that the group $\mu_{p-1}$ of $(p-1)$-st roots of unity is a subgroup of $\Z_p^\times$. For a positive integer $k$ we 
define a subset $S_k\subset \Z_p^\times$ of size $p-1$ by
$$
S_k = \begin{cases} \{\omega - 24 p \omega^2 \ : \ \omega\in\mu_{p-1}\}  & k = 1 \\
\mu_{p-1} & k > 1.
\end{cases}
$$
Suppose $E_{A,B}$ has split multiplicative reduction at $p$. Let $k=\ord_p(\Delta(A,B))>0$. We claim that if $p\nmid k$, then 
$$
 \frac{\Delta(A,B)}{p^k}\!\!\!\!{\pmod{p^2}}\,\not\in\, S_k\!\!\!\!\pmod{p^2}
\iff  \grL(E_{A,B})\in p\Z_p^\times.
$$
To see this, recall that since $E_{A,B}$ has split multiplicative reduction at $p$, it has a Tate paramaterization and a corresponding 
Tate period $q\in p\Zp$,
and $\grL(E_{A,B}) = \frac{\log_pq}{\ord_p(q)}$. The discriminant $\Delta(E_{A,B})\in p\Zp$ is the value of a convergent
series
$$
\Delta(E_{A,B}) = \Delta(q) = q - 24q^2 + \cdots
$$
with $\Z$-coefficients. From this we see that $k = \ord_p(q)$ and that if $p\nmid k$, then $\grL(E_{A,B})\in p\Z_p^\times$
if and only if $\log_pq \in p\Z_p^\times$.  (Here $\log_p$ is the Iwasawa $p$-adic logarithm; in particular, $\log_p p = 0$ and
$\log_p\omega = 0$ for $\omega\in \mu_{p-1}$.)
We can uniquely write $q=p^k\omega u$ for some $\omega\in\mu_{p-1}$ and $u\in 1+p\Zp = (1+p)^{\Zp}$. Then $\log_pq\in p\Z_p^\times$
if and only if $\log_p u \in p\Z_p^\times$, which holds if and only if $u\not\in 1+p^2\Zp$. Since
$$
\frac{\Delta(q)}{p^k} = \omega u  - 24p^k\omega^2u^2 + p^{2k}(\cdots),
$$
we see that if $k>1$, then $u\equiv 1 \pmod{p^2}$ if only if $\frac{\Delta(q)}{p^k} \equiv \omega \pmod{p^2}$. Similarly,
if $k=1$, then $u\equiv 1 \pmod{p^2}$ if and only if $\frac{\Delta(q)}{p} \equiv \omega - 24p\omega^2 \pmod{p^2}$.

Suppose $p\nmid a_0$ and $k=\ord_p(\Delta(a_0,b_0))>0$ (so $p\nmid b_0$). 
Writing $A=a_0+p^ka_1$ and $B=b_0+p^kb_1$, and letting $\Delta_0 = \frac{\Delta(a_0,b_0)}{p^k}$, we find that 
$$
\frac{\Delta(A,B)}{p^k} \equiv \Delta_0 -16\begin{cases}
12 a_0^2 a_1 + 54b_0b_1 + p(12a_0a_1^2 + 27b_1^2) & k=1 \\
12 a_0^2a_1 + 54b_0b_1 &  k>1 \end{cases} \pmod{p^2}.
$$
From this it easily follows that, given $A$ modulo $p^{k+2}$ belonging to any one of the $\frac{1}{2}(p-1)p^{k+1}$ primitive residue classes
modulo $p^{k+2}$ such that $27B^2\equiv -4A^3\pmod p$ is solvable, there exist exactly $(2p-1)(p-1)$ residue classes for $B$ 
modulo $p^{k+2}$ (depending on $A$ modulo $p^{k+2}$) such that $\ord_p(\Delta(A,B))=k$ 
and $\frac{\Delta(A,B)}{p^k}\not\in S_k \pmod{p^2}$.
   In particular, the measure of the set of $(A,B)\in\Z_p^2$ such that 
$p\nmid A$, $\ord_p(\Delta(A,B))=k$ for a given $k>0$ with $p\nmid k$, $E_{A,B}$ has split multiplicative reduction at $p$, and 
$\frac{\Delta(A,B)}{p^k}\not\in S_k \pmod{p^2}$ is equal to $\frac{(p-1)^2(2p-1)}{4p^{k+3}}$ if $p\equiv 1\pmod 4$ and
to $\frac{(p-1)^2(2p-1)}{2p^{k+3}}$ if $p\equiv 3\pmod 4$.

Since $S_1'(5)$ is defined only by congruence conditions modulo powers of $5$, we have 
$$
\mu(S_1'(5)) = \mu_5(S_1'(5)).
$$
We define three sets $\Sigma_5^{\g}$, $\Sigma_5^{\ns}$, and $\Sigma_5^{\spl} \subset \Z_5^2$: 
\begin{itemize}
\item $\Sigma_5^\g$ is the set of those $(A,B)\in \Z_5^2$ with $5\nmid A$ and $5\nmid \Delta(A,B)$ (this is the set
of $(A,B)$ such that $E_{A,B}$ has good ordinary reduction at $5$);
\item $\Sigma_5^{\ns}$ is the set of those $(A,B)\in\Z_5^2$ with $A\equiv 2 \pmod 5$ and $B\equiv \pm 2\pmod 5$ (so $5\mid\Delta(A,B)$)
and $5\nmid \ord_5(\Delta(A,B))$ (this is the set of $(A,B)$ such that $E_{A,B}$ has non-split multiplicative reduction at $5$ and
$5\nmid \ord_5(\Delta(A,B))$);
\item $\Sigma_5^{\spl}$ is the set of those $(A,B)\in\Z_5^2$ with $A\equiv 3 \pmod 5$ and $B\equiv \pm 1 \pmod 5$ (so $5\mid\Delta(A,B)$),
$5\nmid k = \ord_5(\Delta(A,B))$, and $\frac{\Delta(A,B)}{5^k} \pmod {5^2} \not \in S_k \pmod {5^2}$ (this is the set of $(A,B)$ with 
$E_{A,B}$ having split multiplicative reduction at $5$, $5\nmid\Delta(A,B)$, and $\ord_5(\grL(E_{A,B}))=1$).
\end{itemize}
Then $\mu_5(S_1'(5)) = \mu_5(\Sigma_5^\g)+\mu_5(\Sigma_5^\ns)+\mu_5(\Sigma_5^{\spl})$.

Clearly, 
\begin{equation*}
\mu_5(\Sigma_5^\g) = \frac{16}{25}\cdot \left(1-\frac{1}{5^{10}}\right)^{-1}.
\end{equation*}
As was explained above, for any $A\equiv 2,3 \pmod 5$, there are $2(5-1)$ residue classes modulo $5^{k+2}$ (depending on $A$ modulo $5^{k+2}$) such that: $\ord_5\Delta(A,B)=k$ if and only if $B$ belongs to one of these residue classes. Hence
 $$
 \mu_5(\Sigma_5^{\ns}) = \left(1-\frac{1}{5^{10}}\right)^{-1}\cdot \sum_{\scriptstyle{k=1}\atop \scriptstyle{5\nmid k}}^\infty \frac{2(5-1)}{5^{k+2}} = \frac{2}{25}\left(1-\frac{4}{5^5-1}\right)
 \left(1-\frac{1}{5^{10}}\right)^{-1},
 $$
 and
 $$
 \mu_5(\Sigma_5^{\spl}) = \left(1-\frac{1}{5^{10}}\right)^{-1}\cdot \sum_{\scriptstyle{k=1}\atop \scriptstyle{5\nmid k}}^\infty \frac{(5-1)(2\cdot 5-1)}{5^{k+3}} = \frac{9}{125}
 \left(1-\frac{4}{5^5-1}\right)\left(1-\frac{1}{5^{10}}\right)^{-1}.
 $$
 Therefore,
 $$
 \mu_5(S_1'(5)) = \left( \frac{99}{125} - \frac{19}{125}\cdot\frac{4}{5^5-1}\right) \left(1-\frac{1}{5^{10}}\right)^{-1} = .7918054\ldots .
$$
\end{proof}

\begin{lemma}\label{S1-density}
We have $\mu(S_1(5)) > .7917957$.
\end{lemma}

\begin{proof}
The density of $S_1(p)$ equals the density of $S_1'(p)$ times the product over all primes $\ell\equiv \pm 1\pmod p$
of the local densities $\mu_\ell(\Sigma_\ell)$ where $\Sigma_\ell$ is the set of $(A,B)\in \Z_\ell^2$ such that 
$p\nmid \ord_\ell(\Delta(A,B))$ whenever $\ell\mid \Delta(A,B)$. 

We note that for a prime $\ell\geq 5$, the measure of the set of $(A,B)\in\Z_\ell^2$ that satisfy $\ell\nmid A$
and $\ord_\ell (\Delta(A,B))=k$ for a given integer $k>0$ is $\frac{(\ell-1)^2}{\ell^{k+2}}$.
For given an $A$ belonging to one of the $\frac{\ell-1}{2}$ primitive residue classes modulo $\ell$ for which 
$27B^2\equiv -4A^3\pmod \ell$ is solvable in $B$, we have 
$\ord_\ell(\Delta(A,B))=k$ if and only if $B$ belongs to one of $2(\ell-1)$ residue classed modulo $\ell^{k+1}$ (which may
depend on $A$ modulo $\ell^{k+1}$). It follows that for any prime $\ell\equiv \pm 1 \!\pmod p$, $\mu_\ell(\Sigma_\ell)$ is at least
$$
\left(1-\frac{1}{\ell^{10}}\right)^{-1} \left(1  - \sum_{n=1}^\infty \frac{(\ell-1)^2}{\ell^{5n+2}} - \frac{1}{\ell^5}\right)  
= \left(1-\frac{1}{\ell^{10}}\right)^{-1} \left(1  - \frac{(\ell-1)^2}{\ell^2(\ell^5-1)} - \frac{1}{\ell^5}\right).
$$
The first term being subtracted in the second factor on the left-hand side is the measure of those $(A,B)\in \Z_\ell^2$ such that $\ell\nmid A$, $\ord_\ell(\Delta(A,B))>0$, and $5\mid\ord_\ell(\Delta(A,B))$. The second term being subtracted is the measure of the set of $(A,B)$ such that $\ell^2\mid A$ and $\ell^3\mid B$, 
which contains the set of the $(A,B)$ with $\ell^4\mid A$ and $\ell^6\mid B$ {\it and} those $(A,B)$ such that $\ell \mid A$, $\ord_\ell(\Delta(A,B))>0$, and $5\mid \ord_\ell(\Delta(A,B))$.  

Therefore,
$$
\mu(S_1(5)) \geq \mu(S_1'(5))\cdot \prod_{\ell\equiv\pm 1\!\pmod 5} \left(1-\frac{1}{\ell^{10}}\right)^{-1} \left(1  - \frac{(\ell-1)^2}{\ell^2(\ell^5-1)} - \frac{1}{\ell^5}\right)   =  .7917957\ldots .
$$
\end{proof}

\noindent
In particular, we note that
\begin{equation}\label{mu-diff}
\mu(S_1'(5)) - \mu(S_1(5)) \leq .00001.
\end{equation}

Finally, we end with a lemma which explains why we only used Condition (a) of Theorem~\ref{crit0-thm1} to define $S_0(p)$, and only used Conditions (a), (c), and (e) of Theorem~\ref{crit1-thm1} to define $S_1(p)$: namely, $100\%$ of elliptic curves over $\Q$ satisfy Conditions~(b) and (c) of Theorem~\ref{crit0-thm1} and Conditions~(b) and~(d) of Theorem~\ref{crit1-thm1}.

\begin{lemma}\label{lemma0}
Let $p$ be any prime.  Then, when ordered by height, a density of $100\%$ of elliptic curves~$E$ over~$\Q$ possess the following two properties:
\begin{itemize}
\item
$E[p]$ is an irreducible $\Gal(\bar\Q/\Q)$-module;
\item
There exist at least two prime factors $\ell\mid\mid N(E)$, $\ell\neq p$,  such that $E[p]$ is ramified at $\ell$.
\end{itemize}
\end{lemma}

\begin{proof}
That $100\%$ of elliptic curves satisfy the first property follows easily, e.g., by Hilbert irreducibility.  (See also
the work of Duke~\cite{Duke}, who shows that in fact $100\%$ of all elliptic curves~$E$ over~$\Q$, when ordered by height,
  have the property that the action of $\Gal(\bar\Q/\Q)$ on $E[p]$ is
  irreducible for {all} primes $p$.)  
  
  To see that 100\% of elliptic curves over $\Q$ also satisfy the second property, we will show that the density of elliptic curves 
  not satisfying this property is zero.  For a large integer $L$ and for each prime $5\leq \ell\leq L$, $\ell\neq p$, let $S(L,\ell)$ 
  be the set of elliptic curves $E_{A,B}$ such that $\ord_\ell(\Delta(A,B))\leq 1$ and for all primes $5\leq q\leq L$ with $q\notin\{\ell,p\}$, we have 
  $\ord_q(\Delta(A,B))\neq 1$. Let $S(L) = \cup_{5\leq \ell\leq L, \ell\neq p} S(L,\ell)$. 
  Any curve in the complement of $S(L)$ satisfies the second property of the lemma with two primes $5\leq \ell_1,\ell_2\leq L$.
  So it suffices to show that $\mu(S(L))$ tends to $0$ as $L$ gets large.
  
  The set $S(L)$ is a disjoint union of the large sets $S(L,\ell)$, and the density $\mu(S(L,\ell))$ of $S(L,\ell)$ is a product of local densities. Calculating these local densities as in the proof of 
  Lemma~\ref{S1-density}, we find
  $$
  \mu(S(L)) = \sum_{5\leq \ell\leq L,\ell\neq p} (1-\ell^{-10})^{-1}(1-\frac{1}{\ell^2}-\frac{(\ell-1)}{\ell^3})
  \prod_{5\leq q\leq L, q\neq \ell,p} (1-q^{-10})^{-1}(1-\frac{(q-1)^2}{q^3}),
  $$
  which tends to $0$ as $L$ tends to $\infty$.   
\end{proof}

\subsection{A lower bound on the proportion of curves having algebraic and analytic rank~$0$}

We first treat the case of rank $0$.  
 
\begin{theorem}\label{rank0}
  Suppose $F$ is a finite union of large families of elliptic curves such that
  exactly $50\%$ of the curves in~$F$, when ordered by height, have
  root number $+1$.  Then 
  at least $3/8$ of the curves in $F$, when
  ordered by height, have $5$-Selmer rank $0$.  
\end{theorem}

\begin{proof}
Let $p=
5$.  Then by Theorem~\ref{thBS}, the average size of the $p$-Selmer group of curves in~$F$ is $p+1$. On the other hand, by Theorem~\ref{thDD}, we know that that $1/2$ of the curves in $F$ have odd $p$-Selmer rank and thus have at least $p$ elements in the $p$-Selmer group. Hence the $\limsup$ of the 
average size of the $p$-Selmer group among the half 
of elliptic curves in $F$ having even $p$-Selmer rank is at most
$p+2$.  Now if the $p$-Selmer group of an elliptic curve has
even rank, then it must have order 1, $p^2$, or more than $p^2$.  Since the
$\limsup$ of the average of such orders is at most $p+2$, a lower density of at least $(p^2-p-2)/(p^2-1)=(p-2)/(p-1)$ of these orders must be equal to $1$.  Thus
among these $1/2$ of curves in $F$ with even $p$-Selmer
rank, a lower density of at least $(p-2)/(p-1)$ have trivial $p$-Selmer group; i.e., a lower density of at least $1/2\cdot (p-2)/(p-1) = (p-2)/(2p-2)$ of curves in $F$ have $p$-Selmer rank~0.   
For $p=5$, this yields a lower proportion of $3/8$ of curves in~$F$ having $5$-Selmer rank 0.
\end{proof}

\begin{corollary}\label{rank0cor}
When ordered by height, at least $16.50\%$ of elliptic curves over $\Q$ have both algebraic and analytic rank~$0$.
\end{corollary}

\begin{proof}
By definition, a proportion of $\mu(S_0(5))$ of all elliptic curves over $\Q$ satisfy condition (a) of Theorem~\ref{crit0-thm1}.  By Theorem~\ref{equithm}, inside the family $S_0(5)$ there exists a finite union $F$ of large subfamilies of density $\mu(F)=\kappa\cdot \mu(S_0(5))$,
with $\kappa\geq .5501$, and having equidistributed root number.  By Theorem~\ref{rank0}, a proportion of 3/8 of elements of $F$ have $5$-Selmer rank 0.  Thus, by Theorem~\ref{crit0-thm1} and Lemma~\ref{lemma0}, a proportion of at least $3/8\cdot \kappa\cdot \mu(S_0(5))$ of all elliptic curves over $\Q$ have both algebraic and analytic rank 0.

By Theorem~\ref{equithm} and Lemma~\ref{S0-density}, we thus obtain a lower density of greater than 
$$
\frac{3}{8}\times .5501 \times .8  = .16503
$$ 
of all curves have both algebraic and analytic rank 0, as stated in the corollary.
\end{proof}

\noindent

In other words, at least $16.50\%$ of all elliptic curves have rank 0 and satisfy the Birch and Swinnerton-Dyer conjecture.  

\subsection{A lower bound on the proportion of curves having algebraic and analytic rank~1}

We are now ready to treat the case of rank 1.
 
\begin{theorem}\label{rank1}
  Suppose $F$ is a finite union of large families of elliptic curves such that
  exactly $50\%$ of the curves in~$F$, when ordered by height, have
  root number $+1$.  Then 
  at least $19/40$ of the curves in $F$, when
  ordered by height, have $5$-Selmer rank $1$.  
\end{theorem}

\begin{proof}
Let $p=5$.  Again, by Theorem~\ref{thBS}, the average size of the $p$-Selmer group of curves in~$F$ is $p+1$; by Theorem~\ref{thDD}, we know that that $1/2$ of the curves in $F$ have even $p$-Selmer rank.
Since every $p$-Selmer group has at least one element (namely, the identity element), the limsup of the 
average order of the $p$-Selmer groups among the half 
of elliptic curves in $F$ that have odd $p$-Selmer rank is at most
$2p+1$.  Now if the $p$-Selmer group of an elliptic curve has
odd rank, then it must have order $p$, $p^3$, or more than $p^3$.  Since the limsup of the 
average of such orders is at most $2p+1$, a lower density of at least $(p^3-2p-1)/(p^3-p)=(p^2-p-1)/(p^2-p)$ of these orders must be equal to $p$.
Thus
among these $1/2$ of curves in $F$ with odd $p$-Selmer
rank, at least $(p^2-p-1)/(p^2-p)$ have $p$-Selmer rank 1;  i.e., a lower density of at least $1/2\cdot (p^2-p-1)/(p^2-p) = (p^2-p-1)/(2p^2-2p)$ of curves in $F$ have $p$-Selmer rank~1.   
For $p=5$, this yields a lower proportion of $19/40$ of curves in~$F$ having $5$-Selmer rank 1.
\end{proof}

\begin{corollary}\label{rank1cor}
When ordered by height, at least $20.68\%$ of elliptic curves over $\Q$ have both algebraic and analytic rank $1$.
\end{corollary}

\begin{proof}
By definition, a proportion of $\mu(S_1'(5))$ of all elliptic curves over $\Q$ satisfy Conditions (a) and (e) of Theorem~\ref{crit1-thm1}.  
By Theorem~\ref{equithm}, there exists a finite union $F$ of large subfamilies in $S_1'(5)$ with density 
$\kappa\cdot \mu(S_1'(5))$, with $\kappa\geq.5501$, and having equidistributed root numbers.  
By Theorem~\ref{rank1}, a proportion of at least 19/40 of the elliptic curves $E$ in $F$ have the property that $\#S_5(E)=5$.
We conclude by Theorem~\ref{crit1-thm1}, Lemma~\ref{lemma0}, and \eqref{mu-diff} that a proportion of at least 
$$
\frac{19}{40} \times .5501 \times  .7918054 - .00001 > .20688
$$ 
of all elliptic curves lie in $S_1(5)$ and have both algebraic and analytic rank 1.
\end{proof}

\noindent

In other words, at least $20.68\%$ of all elliptic curves have rank 1 and satisfy the Birch and Swinnerton-Dyer conjecture.

\subsection{A lower bound on the proportion of curves having algebraic and analytic rank~0 or 1}

Corollaries \ref{rank0cor} and \ref{rank1cor} together already show that a proportion of at least $16.50+20.68=37.18\%$ of all elliptic curves satisfy the Birch and Swinnerton-Dyer rank conjecture.  However, we can do much better if we do not consider the individual ranks 0 and 1 separately.  Specifically, we can prove that a large proportion of curves have algebraic and analytic rank 0 or 1 even in large families where the root number is not equidistributed:
 
\begin{theorem}\label{rank0or1}
Let $F\subset S_0(5)\cap S_1(5)$ be any finite union of large families of elliptic curves.  Then at least 
$19/24$ of the elliptic curves in $F$ have algebraic and analytic rank $0$ or $1$.  
If, furthermore, root numbers are equidistributed in $F$,
then at least 
$7/8$
of the elliptic curves in $F$ have algebraic and analytic rank $0$ or $1$.
\end{theorem}

\begin{proof}
Let $p=5$. By Theorem~\ref{thBS}, the average size of the $p$-Selmer group of curves in $F$ is $p+1$.  
Let $x_{0{\rm \,or\,}1}$ be the lower density of elliptic curves in $F$ having
$5$-Selmer rank $0$ or $1$. Then
$$
x_{0{\rm \,or\,}1}+p^2(1-x_{0{\rm \,or\,}1})\leq p+1,
$$ and hence $x_{0{\rm \,or\,}1}\geq (p^2-p-1)/(p^2-1)$. The bound is achieved when a proportion
of $(p^2-p-1)/(p^2-1)$ of elliptic curves in $F$ have $p$-Selmer rank $0$, and a
proportion of $p/(p^2-1)$ have $p$-Selmer rank $2$.
Setting $p=5$, and then applying Theorem~\ref{crit1-thm1}, now proves the first part of  Theorem~\ref{rank0or1}.

To prove the second part of Theorem~\ref{rank0or1}, let $x_{0{\rm \,or\,}1}$ again denote the lower density of elliptic curves with $5$-Selmer rank 0 or 1.  
Also, let $x_0$ (resp.\ $x_1$) denote the lower density of elliptic curves
with $5$-Selmer rank $0$ (resp.\ $1$).  
Then, by Theorem~\ref{thBS} and~\ref{thDD}, we have
$$ x_0+p^2(1/2-x_0)+p(x_1+p^2(1/2-x_1))\leq p+1.$$ Thus, we obtain
$$(p^2-1)x_0+(p^3-p)x_1\geq (p^3+p^2)/2-p-1.$$ In conjunction with the constraint $x_1\leq 1/2$, it
follows that $$x_{0{\rm \,or\,}1}\geq x_0+x_1\geq [(p^2+p)/2-p-1]/(p^2-1)+1/2=(2p-3)/(2p-2).$$ 
Again, this bound is achieved when a proportion of $(p-2)/(2p-2)$ of elliptic curves over $\Q$ have $p$-Selmer rank $0$, a proportion of 
$1/2$ of elliptic curves have $5$-Selmer rank $1$, and a proportion of $1/(2p-2)$ of
elliptic curves have $5$-Selmer rank $2$.  Setting $p=5$, and applying Theorem~\ref{crit1-thm1}, 
now yields the second part of Theorem~\ref{rank0or1}.
\end{proof}

\begin{corollary}\label{rank0or1cor}
When ordered by height, at least $66.48\%$ of elliptic curves over $\Q$ have algebraic and analytic rank $0$ or $1$.
\end{corollary}

\begin{proof}
By Theorem~\ref{equithm} there is a finite union $F'$ of large subfamilies in $S_0(5)\cap S_1'(5)$ of 
density $\kappa\mu(S_0(5)\cap S_1'(5))$ with $\kappa\geq.5501$ and such that for all $E\in F'$ the root number of $E$ and its $-1$-twist
have opposite signs. Let $F = F'\cap S_1(5)$. Then $F$ is also a finite union of large subfamilies and, since $S_1(5)$ is 
stable under $-1$-twist, $F$ also has the property that for all $E\in F$ the root number of~$E$ and its $-1$-twist
are both in $F$ and have opposite signs. In particular, the root numbers of the curves in $F$ are equidistributed. 
By \eqref{mu-diff}, the density of $F$ satisfies
$$
\mu(F) \geq \kappa\mu(S_0(5)\cap S_1'(5)) - .00001.
$$
By  the second part of Theorem~\ref{rank0or1}, a proportion of at least $7/8$ of the curves in $F$
have algebraic and analaytic rank 0 or 1.

Next we consider the set $F''$ of curves in $S_0(5)\cap S_1(5)$ on which the above arguments have not been applied.
This set contains the complement of $F$ in $S_0(5)\cap S_1(5)$, which is a (possibly infinite) disjoint union of large subfamilies.  
So for any $\epsilon > 0$, $F''$ contains
a finite union $F''_\epsilon$ of large subfamilies such that 
$$
\mu(F''_\epsilon) \geq \mu(S_0(5)\cap S_1(5)) - \mu(F) - \epsilon\geq (1-\kappa) \cdot \mu(S_0(5)\cap S_1'(5)) - .00001 - \epsilon.
$$
By Theorem~\ref{rank0or1}, a proportion of at least $19/24$ of the curves in $F''_\epsilon$ have algebraic
and analytic rank 0 or 1.

For the set of elliptic curves in $S_0(5)$ on which the above arguments have not been applied, which has density at least $.8-.79179\ldots
=.00820\ldots$,
we can apply the arguments of Corollary \ref{rank0cor}. This gives an additional set of curves of density at least 
$3/8\times .5501\times .00820 = .00169\ldots$ that have algebraic and analytic rank $0$.

It follows that a total proportion of at least
$$
\left(\frac{7}{8}\kappa + \frac{19}{24}(1-\kappa)\right)\times \mu(S_0(5)\cap S_1(5)') - \left(\frac{7}{8}+\frac{19}{24}\right)\times .00001 + 
.00169
$$
of elliptic curves have algebraic and analytic rank $0$ or $1$. Since $\kappa\geq .5501$, we conclude by Lemma~\ref{S1prime-density} that
this proportion is at least
$$
\left(\frac{7}{8}\times .5501 + \frac{19}{24}\times .4499\right)\times .7918054\ldots - \left(\frac{7}{8}+\frac{19}{24}\right)\times .00001 + 
.00169 = .664816\ldots.
$$
\end{proof}

\noindent

Thus at least $66.48\%$ of all elliptic curves over $\Q$ satisfy the Birch and Swinnerton-Dyer conjecture, yielding Theorem~\ref{bsdcor}.

\section{Conclusion and future work}

The proportions in Corollaries \ref{rank0cor}, \ref{rank1cor}, and \ref{rank0or1cor} can be improved by strengthening Theorems \ref{crit0-thm1},
\ref{crit1-thm1}, \ref{thBS}, or \ref{equithm}. In particular, if Theorem \ref{crit0-thm1} is improved as indicated in Remark~\ref{rmk0}, then 
the lower bound on the proportion of elliptic curves having algebraic and analytic rank $0$ increases to $19.8\%$, and the lower bound on the proportion of elliptic curves having algebraic and analytic rank $0$ {\it or} algebraic and analytic rank $1$ increases to $69.6\%$. If Theorem \ref{crit1-thm1} can
be improved as indicated in Remark~\ref{rmk1}, then the lower bound on the proportion of elliptic curves having algebraic and analytic rank $1$
increases to $24.8\%$, and the lower bound on the proportion of elliptic curves 
satisfying the Birch and Swinnerton-Dyer rank conjecture increases to $79.7\%$ (working also with the prime $3$ would push this lower bound over $80\%$).

Finally, there have been a number of recent heuristics (cf.\ Delaunay~\cite{Delaunay}, Poonen--Rains~\cite{PR}, and~\cite{BS5}) that independently suggest that, for all $p$, the average size of the $p$-Selmer group of all elliptic curves, when ordered by height, should be $p+1$.  Indeed, Theorem~\ref{thBS} gives excellent evidence for this conjecture, confirming the conjecture for the primes $p=2$, 3, and 5.

Tracing through the methods of the previous sections, with a general value of $p$ in place of $p=5$, immediately allows us to deduce: 

\begin{theorem}\label{allpavg}
Consider the family of all elliptic curves over $\Q$ ordered by height. 
Suppose that for 
all primes $p$, 
the average size of the $p$-Selmer group of elliptic curves over $\Q$ is $p+1$.  
Then the Birch and Swinnerton-Dyer rank conjecture is true for $100\%$ of elliptic curves over~$\Q$. 
\end{theorem}

\begin{proof} Let $p$ be a prime. A lower bound on the density $\mu(S_0(p)\cap S_1(p))$ is 
\begin{equation}\label{pdensity}
(1-1/p)^2\cdot\prod_{\ell\equiv\pm 1\!\!\!\!\pmod p} (1-1/\ell^{10})^{-1}(1-1/\ell^5)^2.
\end{equation}
Of these, by the proof of Corollary \ref{rank0or1cor}, a proportion of at least 
\begin{equation}\label{0or1prop}
\frac{p^2-p-1}{p^2-1}
\end{equation} 
have algebraic and analytic rank $0$ or $1$. As $p$ tends to infinity, the product of \eqref{pdensity} and \eqref{0or1prop} tends to~$1$.
\end{proof}

Indeed, in this paper, we have used the $p$-Selmer average for the prime $p=5$ to prove unconditionally that at least $66.48\%$ of elliptic curves over $\Q$ satisfy the Birch and Swinnerton-Dyer rank conjecture.

\subsection*{Acknowledgments}

We are very grateful to Arul Shankar and Xin Wan for helpful conversations. 
The first named author was
supported in part by National Science Foundation Grant~DMS-1001828 and
a Simons Investigator Grant.  The second named author was supported in
part by National Science Foundation Grants~DMS-0701231 and
DMS-0758379.  The third named author was supported in
part by a National Science Foundation Grant~DMS-1301848 and a Sloan Research Fellowship.

\end{document}